\documentclass[12pt,reqno]{article}
\usepackage{amssymb,amsmath}
\usepackage{amsthm}
\usepackage[all,cmtip]{xy}
\usepackage{color}
\usepackage[colorlinks=true,
filecolor=webbrown,
citecolor=webwheel]{hyperref}
\definecolor{webwheel}{rgb}{0,.5,0}
\definecolor{webbrown}{rgb}{.6,0,0}
\usepackage{fullpage}
\usepackage{float}
\usepackage{graphics,amsmath,amssymb}
\usepackage{amsthm}
\usepackage{amsfonts}
\usepackage{latexsym}
\usepackage{epsf}
\begin{document}
\theoremstyle{plain}
\newtheorem{theorem}{Theorem}
\newtheorem{corollary}[theorem]{Corollary}
\newtheorem{lemma}[theorem]{Lemma}
\newtheorem{proposition}[theorem]{Proposition}

\theoremstyle{definition}
\newtheorem{definition}[theorem]{Definition}
\newtheorem{example}[theorem]{Example}
\newtheorem{conjecture}[theorem]{Conjecture}

\theoremstyle{remark}
\newtheorem{remark}[theorem]{Remark}

\baselineskip17pt

\newcommand{\A}{\mathbb{N}_0}
\newcommand{\ber}{\mathcal{B}}
\newcommand{\eulr}{\mathcal{E}}
\newcommand{\s}{\mathcal{S}}
%\numberwithin{equation}{section}
\title{\LARGE\bf On an Arithmetic Convolution}
\date{}
\author{Jitender Singh\\Department of Mathematics\\ Guru Nanak Dev
University\\ Amritsar-143005, Punjab, INDIA\\
\tt sonumaths@gmail.com\\ https://sites.google.com/site/sonumaths2/}
\maketitle
\begin{abstract}
The Cauchy-type product of two arithmetic functions $f$ and $g$ on nonnegative integers is defined as
$(f\bullet g)(k):=\sum_{m=0}^{k} {k\choose m}f(m)g(k-m)$.
We explore some algebraic properties of the aforementioned convolution, which is a fundamental-characteristic of the identities involving the Bernoulli numbers, the Bernoulli polynomials, the power sums, the sums of products,  etc.
\end{abstract}
\section{Introduction}
An arithmetic function  is a map $f:\A\rightarrow \mathbb{C}$. Let $\psi:\A\times \A \rightarrow \A$ be a binary operation
satisfying the following properties:

(1) the binary operation $\psi$ is associative and commutative

(2) for each $n\in \A,$ the set $\{(x,y)\in\A\times \A~|~\psi(x,y)=n,~n\geq 0\}$ is finite

(3) either $\psi(x,0)=n$ implies $x=n$ or $\psi(x,1)=n$ implies $x=n$.

For a given $\psi,$ as defined above, a $\psi-$convolution on the set $\s$ of all arithmetic functions is a binary operation $*_\psi:\s\times \s\rightarrow \s,$ such that $(f*_\psi g)(k)=\sum_{\psi(k_1,k_2)=k}f(k_1)g(k_2)$ for all $f, g\in\s$.

The Cauchy product with $\psi(k_1,k_2)=k_1+k_2;$ the Dirichlet product with $\psi(k_1,k_2):=k_1k_2$ on $\mathbb{N};$ the natural product with
$$\psi(k_1,k_2)=\begin{cases}
                         k_1,& \text{if}~ k_1=k_2;\\
                          0, &\text{otherwise}.\end{cases}$$
and the lcm product with $\psi(k_1,k_2)=\text{lcm}(k_1,k_2)$
are some known $\psi-$convolutions. For the other convolutions,
the reader may refer to Lehmer \cite{lehmer}, Subbarao \cite{subbarao}, McCarthy \cite{McCarthy}, and references cited therein.

 The purpose of the present study is to explore some algebraic properties of the weighted Cauchy product on  $\mathcal{S}$ defined as follows:
 \begin{equation}\label{eq1}
(f\bullet g)(k):=\sum_{\psi(k_1,k_2)=k}{k_1+k_2\choose k_1}f(k_1)g(k_2),
 \end{equation}
where $\psi(k_1,k_2)=k_1+k_2$.  With the usual addition, $(f+g)(k)=f(k)+g(k),$ for all $k\in\A$ and $f,g\in\mathcal{S},$ the triple $(\mathcal{S},+,\bullet)$ is a commutative ring with one. The convolution $\bullet,$ also known as binomial convolution, was used by Haukkanen \cite{Hau} to discuss the roots of arithmetic functions. The convolution, however, has largely remained unnoticed. The term ``binomial convolution'' is also used to denote another convolution in the modern literature \cite{toth1}, so,  in the sequel,
we will call the aforementioned convolution $\bullet$, \emph{the Cauchy-type product}.
Recently, Gould and Quaintance  \cite{gould} have used the term binomial-transform  for the Cauchy-type product.

Consider the set $\mathcal{A}:=\{f\in\s~|~f(0)\neq 0\}$ with the Cauchy-type product as binary operation. Then the set $\mathcal{A}$ with the identity $e$ defined by
$$e(k):=\begin{cases}
                             1, \text{if}~ k=0;\\
                             0, \text{otherwise}.\end{cases}$$
and, the inverse $f^{-1}=g$ defined inductively by
\begin{equation}\label{eq2}
g(0)=\frac{1}{f(0)};~g(k)=-\frac{1}{f(0)}\sum_{m=1}^{k-1}{k\choose m}f(m)g(k-m),
\end{equation}
is an abelian group as $\mathcal{A}$ does under the well-known Cauchy product.
It also follows from  \eqref{eq2} that, $f\in\s$ is invertible if and only if $f(0)\neq 0$. Therefore,
$\mathcal{A}$ serves as the group of units in the ring $\mathcal{S}$.
\begin{definition}
Let $I,\nu\in\mathcal{A}$ be the arithmetic functions such that $I(k)=1$ and  $\nu(k)=(-1)^k,$ for all $k\in \A$.
\end{definition}
Note that $I\bullet \nu=e$, which gives $I^{-1}=\nu$.
The following easy result shows that, in the Cauchy-type product, the arithmetic function
$\nu$ plays a role analogous to that of the M\"obius function.
\begin{proposition}
Let $f\in\mathcal{A},$ and $F(k):=\sum_{m=0}^k {k\choose m}f(m),$
$k\in\A$. Then $F\in\mathcal{A},$ and $f(k)=\sum_{m=0}^k {k\choose m}F(m)(-1)^{k-m}$.
\end{proposition}
\begin{proof}
    It follows from the fact that $F=f\bullet I$ and $I\bullet \nu=e$.
\end{proof}
Rest of the paper is organized as follows.
The identities concerning the Bernoulli numbers,
the Bernoulli polynomials, and the power sums are discussed via the Cauchy-type product in Section \ref{sec:2}.
In the same Section, some interesting properties of the arithmetic functions, which obey symmetric identities are also derived.
It will become apparent that the identities are a consequence of the Cauchy-type product. Algebraic properties of the Cauchy-type product
are obtained in Section \ref{sec:3}.
In Section \ref{sec:4}, some identities involving both
the Dirichlet product and the Cauchy-type product are discussed.
\section{Motivation}\label{sec:2}
The Cauchy-type product is inherent in many identities involving binomial coefficients.
For example, it appears in the following fundamental identity of the Bernoulli numbers
\begin{equation}\label{eq3}
   \sum_{m=0}^{k}{k\choose m}\frac{B_m}{k+1-m}=e(k)~\text{for all}~k\in\A,
\end{equation}
where the $k$th Bernoulli number is denoted by $B_k$.
The left side of \eqref{eq3} is the Cauchy-type product
$\ber\bullet \xi_1,$ where $\ber(k):=B_k$ and $\xi_1(k):=\frac{1}{k+1}$. Thus \eqref{eq3} is equivalent to $\ber=\xi_1^{-1}$.

A relatively simple identification is the binomial theorem,
$\sum_{m=0}^{k}{k\choose m}x^m=(\epsilon_x\bullet I)(k)$,
where  $\epsilon_0:=e$ and $\epsilon_1=I,$ so that $\epsilon_x(0)=1$ and $\epsilon_x(k):=x^k$ for $k>0$.
It is easy to verify that  $\epsilon_x\in \mathcal{A},$ for each $x\in\mathbb{C},$ where
\begin{equation}
    \epsilon_x^{-1}(k)=\begin{cases}
    1,~\hfill \text{if}~ k=0; \\
    (-x)^k,~\text{if}~ k>0.
    \end{cases}
\end{equation}
Further, the $k$th Bernoulli polynomial $\ber_x(k)$ is defined by the generating function $\displaystyle \frac{te^{xt}}{e^t-1}=\sum_{k=0}^{\infty}\ber_x(k)\frac{t^k}{k!},$ where $x\in\mathbb{C}$ and $k\in\A$.
We then have $\ber(k)=\ber_0(k)$. Among many other properties of $\ber_x$ and $\ber$, we are interested in  the one, which is given by
$\ber_x(k)=\sum_{m=0}^{k}{k\choose m}x^m\ber(k-m)$
or $\ber_x=\ber\bullet \epsilon_x$. This identification
leads to the identities,
$\ber_{x+y}=\ber\bullet \epsilon_{x+y}=(\ber\bullet \epsilon_{x})\bullet \epsilon_{y}=\ber_x\bullet \epsilon_y;~\ber_{1-x}=\nu \ber_x$.
Also, by inversion, $\ber=\ber_x\bullet \epsilon_x^{-1},$  which gives
$$\ber(0)=\ber_x(0);~\ber(k)=\sum_{m=0}^{k}{k\choose m}\ber_x(k-m)(-x)^m,~k\in\mathbb{N}.$$
If we take $x=0$ and $x=1,$  in  $\ber_{1-x}=\nu \ber_x$ and $\ber=\ber_x\bullet \epsilon_{-x}$ respectively, we find that $I\bullet \ber=\nu \ber$.

The inverse of $\ber_x$ with respect to the Cauchy-type product is given by
\begin{equation}\label{eq5}
    \ber_x^{-1}=\xi_1\bullet \epsilon_{-x}=\xi_{1}(\xi_{-x+1,1}-\xi_{-x,1}),
\end{equation}
where $$\xi_{x,m}(k):=\begin{cases}
                 \frac{x^{k+m}}{k+m},~\text{if}~x\neq 0; \\
                 e(k),~~\text{otherwise}.
               \end{cases}$$
and $\xi_{0,m}:=e$ for each $m\in\mathbb{C}^\times$.
By applying the transformation $x\rightarrow -x$ in \eqref{eq5}, we get
\begin{equation}\label{eq6}
    \ber_{-x}^{-1}(k)=\xi_{1}(k)(\xi_{x+1,1}(k)-\xi_{x,1}(k))=\frac{(1+x)^{k+1}-x^{k+1}}{k+1}.
\end{equation}
It can be easily verified that $\displaystyle \xi_1^2(k)=2\frac{2^{k+1}-1}{(k+1)(k+2)}$. Therefore, now \eqref{eq5} implies
\begin{equation}\label{eq7}
  \ber_x^{-2}(k)=-2\frac{(1-2x)^{k+2}-2^{k+1}(1-x)^{k+2}-2^{k+1}(-x)^{k+2}}{(k+1)(k+2)},
\end{equation}
and the process  inductively leads to the following.
\begin{theorem}\label{th0}For any $n\in\mathbb{N}$
\begin{equation}\label{eq8}
\ber_x^{-n}(k)=\frac{k!}{(k+n)!}\sum_{j=0}^{n}\left(\begin{array}{c}
                                                               n \\
                                                               j
                                                             \end{array}\right)
(-1)^{n-j}(j-nx)^{k+n}.
\end{equation}
In addition,  $\ber_{1-x}^{-n}(k)=(-1)^{k}\ber_{x}^{-n}(k)$ and $\frac{d}{dx}\ber_x^{-n}(k)=-n k\ber_x^{-n}(k-1)$.
\end{theorem}

By Theorem \ref{th0}, one obtains a generalization of \eqref{eq3} as
\begin{equation}\label{eq9}
    n!\displaystyle \sum_{m=0}^{k}\sum_{j=0}^{n}\frac{\ber_x^{n}(k-m)}{(k-m)!}\frac{(-1)^{n-j}}{(n-j)!}\frac{(j-nx)^{m+n}}{(m+n)!}\frac{1}{j!}=e(k)~,n\in\mathbb{N}
\end{equation}
which is precisely  \eqref{eq3}  for $n=1$.

It is interesting to note the following two identities obtainable from \eqref{eq9}, for $k=0,$ $n\in\mathbb{N}$ and $x\in\mathbb{C}$
\begin{equation}
\sum_{j=0}^{n}\frac{(-1)^{n-j}}{(n-j)!}\frac{(j-nx)^n}{j!}=1,\end{equation}
\begin{equation}
\sum_{j=0}^{n}\frac{(-1)^{n-j}}{(n-j)!}\frac{(j-nx)^{n-\alpha}}{j!}=0,~\text{for all}~n\geq \alpha,~x\neq \frac{1}{i}, ~i=1,\ldots,\alpha,~\alpha\in\mathbb{N}.
\end{equation}
\subsection{Power sum identities}
The ring structure on $\s$  is helpful in computing many other identities. For example,  from \eqref{eq6} one has
$$\left(\xi_1\bullet\sum_{x=1}^{n}\epsilon_x\right)(k)=
\sum_{x=1}^{n}(\xi_1\bullet\epsilon_x)(k)=\sum_{x=1}^{n}\ber_{-x}^{-1}(k)=\frac{(n+1)^{k+1}-1}{k+1},$$ which allows the
classical Faulhaber formula for the power sum in the form below (see Singh \cite{singh2009})
\begin{equation}\label{eq10}
\sum_{x=1}^{n}\epsilon_x(k)=\sum_{x=1}^n(\xi_1^{-1}\bullet\ber_{-x}^{-1})(k)=\sum_{m=0}^{k}
{k\choose m}\ber(m)\frac{(n+1)^{k+1-m}-1}{k+1-m}.
\end{equation}

Now, if $\mathcal{S}_n(k)$ denotes the power sum $\sum_{x=1}^{n}\epsilon_x(k),$  then \eqref{eq10} can be used to
represent the Faulhaber power sum by  $\mathcal{S}_x:=\ber\bullet \xi_{x+1,1}-e$.
We also have $\ber_{x+1}-\ber_{x}=\ber\bullet \epsilon_{x+1}-\ber\bullet \epsilon_x=\ber\bullet (\epsilon_{x+1}-\epsilon_x)$ using which the power sum can be expressed as
\begin{equation}
    \mathcal{S}_x(k)=(\ber\bullet (\xi_{x+1,1}-\xi_{1,1}))(k)=\xi_1(k)(\ber\bullet(\epsilon_{x+1}-\epsilon_1))(k+1)=\xi_1(k)(\ber_{x+1}-\ber_1)(k+1).
\end{equation}
\subsection{Symmetric identities}
The Bernoulli numbers are also known to satisfy many symmetric identities.
For given  $m,n\in \A$, $\ber$ satisfies the following identity attributed to Carlitz \cite{carlitz1971}
\begin{equation}\label{eq11}
(-1)^n\sum_{i=0}^n{n\choose i}\ber(m+i)=(-1)^m\sum_{i=0}^{m}{m\choose i}\ber(n+i).
\end{equation}
More generally, if  $F:=I\bullet f,~f\in\mathcal{A},$ then (see Gould and Quaintance \cite{gould})
\begin{equation}\label{eq12}
\sum_{i=0}^n{n\choose i}f(m+i)=\sum_{i=0}^{m}{m\choose i}(-1)^{m-i} F(n+i)~\text{for all}~m,n\in\A.
\end{equation}
Letting $f=\ber$ in \eqref{eq12} recovers  \eqref{eq11}.
The identity \eqref{eq12} motivates the following.
\begin{definition}
     For each $m\in \A,$ let $\psi_m:\A\rightarrow \A$. Define $\psi_m-$product of $f,g\in\s$ to be the map $\otimes_{\psi_m}: \s\times \s \rightarrow \s,$ such that $(f\otimes_{\psi_m} g)(n):=((f\circ\psi_m)\bullet (g\circ \psi_m))(n),$ where $\circ$ denotes the composition of maps.
\end{definition}
It follows from the definition that the binary operation $\otimes_{\psi_m}$ is commutative and associative, for each fixed $m\in\A$. Also, if $\psi_m$ is the identity map of $\A,$ then $\otimes_{\psi_m}$ coincides with the Cauchy-type product.

Note that the identity \eqref{eq12} can be rewritten as
\begin{equation}\label{eq13}
    (I\otimes_{\psi_m} f)(n)=(-1)^n(\nu\otimes_{\psi_n} F)(m),
\end{equation}
where  $\psi_m(n):=n+m$.

 The next result tells us that the Bernoulli numbers are not the only ones, which satisfy the symmetric identity \eqref{eq11}.
 \begin{theorem}
    For nonnegative integers $n$ and $m,$ and $f\in\mathcal{A},$
    \begin{equation}\label{eq14}
         (-1)^n(I\otimes_{\psi_m} f)(n)=(-1)^m(I\otimes_{\psi_n} f)(m) ~\Leftrightarrow~ I\bullet f=\nu f,
    \end{equation}
    where $\psi_m(n)=m+n$. Furthermore, any $f\in\mathcal{A}$ satisfying \eqref{eq14} is determined by its image $f(2\A),$ i.e.,
   \begin{equation}\label{eq15}
 f(2k+1)=-\sum_{i=0}^{k}{2k+1\choose 2i+1}\eulr_{1}(2i+1)f(2(k-i)),
   \end{equation}
 where $\eulr_{1}(2i+1)$ denotes the $(2i+1)$th Euler polynomial evaluated at $1$.
 \end{theorem}
\begin{proof}
    We take $I\bullet f=\nu f$ in \eqref{eq12} to obtain the symmetric identity in \eqref{eq14}.

   To prove the converse part, we put $m=0$ in \eqref{eq14}.

For the remaining part,
note that the Euler polynomials are defined via the generating function
$\frac{e^{2xt}}{e^t+1}=\sum_{k=0}^{\infty}E_k(x)\frac{x^k}{k!},$ where
the $k$th Euler polynomial is denoted by $E_k(x)$. Let $\eulr_x\in\mathcal{A}$
be such that $\eulr_x(k):=E_k(x)$. Then $\eulr_1$ satisfies
\begin{equation}\label{eq16}
        \nu\bullet \eulr_1+\eulr_1= 2e,
    \end{equation}
    since $\frac{2e^{t}}{e^t+1}+e^{-t}\frac{2e^{t}}{e^t+1}=2$.

Now let $f\in\mathcal{A}$ satisfies \eqref{eq14}. Multiplying \eqref{eq16} by $I\bullet f$ and using $(I\bullet f)=\nu f,$ we get
\begin{equation}\label{eq17}
    \begin{split}
     \eulr_1\bullet \frac{(\nu f +f)}{2}= \nu f,
    \end{split}
\end{equation}
from which  \eqref{eq15} follows.
\end{proof}
\begin{remark}Any $f\in\mathcal{A}$ satisfying \eqref{eq14} also satisfies
\begin{equation}\label{eq18}
 f(2k)=-\frac{2}{2k+1}\sum_{i=0}^{k}{2k+1\choose 2i+1}\ber(2(k-i))f(2i+1),~k>0,
   \end{equation}
   but \eqref{eq18} is not sufficient to determine $f$.
\end{remark}
If some $f\in\mathcal{A}$ does not satisfy the symmetric identity \eqref{eq14},
we may define the deviation $\Delta_f\in\s$ of $f\in\mathcal{A}$ by
\begin{equation}
    \Delta_f:=I\bullet f-\nu f.
\end{equation}
Using this in \eqref{eq13}, we find that
\begin{equation}\label{eq19}
    (-1)^n(I\otimes_{\psi_m} f)(n)-(-1)^m(I\otimes_{\psi_n} f)(m)=(\nu\oplus_{\psi_n}\Delta_f)(m).
\end{equation}
\begin{example}
   The arithmetic function $\eulr_1$ satisfies $\nu\bullet \eulr_1=\nu\eulr_1$ which, together with \eqref{eq16}, gives
$\Delta_{\eulr_1}=2(I-e)$. Therefore, $(\nu\oplus_{\psi_n}\Delta_{\eulr_1})(m)=-2\nu(m)e(n)$.  Now from \eqref{eq19},
one has the following identity
\begin{equation}\label{eq20}
    (-1)^n(I\otimes_{\psi_m} \eulr_1)(n)-(-1)^m(I\otimes_{\psi_n} \eulr_1)(m)=2(\nu(n)e(m)-\nu(m)e(n)),
\end{equation}
 which, in particular, gives the symmetric identity
$$(-1)^n(I\otimes_{\psi_m} \eulr_1)(n)=(-1)^m(I\otimes_{\psi_n} \eulr_1)(m),~\text{for all}~m,n\in\mathbb{N}.$$
\end{example}
\begin{theorem}
    For $a,b,c\in\mathbb{C},$
    \begin{equation}\label{eq21}
     \epsilon_a\ber_c\bullet \epsilon_b \ber_{a+c}=\epsilon_b\ber_c\bullet \epsilon_a \ber_{b+c}.
    \end{equation}
     Further, if $\sigma_x\in\s$ is such that $(n+1)\sigma_{x}(n):=(\ber_{x+1}-\ber)(n+1)$ for $n\in\A,$ then
    \begin{equation}\label{eq22}
  b(\epsilon_{a}\ber_c\bullet \epsilon_b\sigma_{a+c-1})=a(\epsilon_{b}\ber_c\bullet \epsilon_a\sigma_{b+c-1}).
   \end{equation}
\end{theorem}
\begin{proof}
    Observe that $\ber_{x+c}=\epsilon_x\bullet\ber_c$  and $\epsilon_x (\epsilon_y f\bullet \epsilon_z g) =\epsilon_{xy}f\bullet \epsilon_{xz} g,$ for  $x,c\in\mathbb{C}$  and $f,g\in \s$.  Therefore,
   $\epsilon_a\ber_c\bullet \epsilon_b \ber_{a+c}=\epsilon_{a}\ber_c\bullet\epsilon_b\ber_c\bullet \epsilon_{ab}$.
   The right-side being symmetric in $a$ and $b,$ establishes \eqref{eq21}.

   For proof of the second part, we use $\ber_{x}(n)=n\sigma_{x-1}(n-1)+\ber(n)$ for $x=a+c,~b+c$ in \eqref{eq21} to obtain
   $$nb(\epsilon_{a}\ber_c\bullet \epsilon_b\sigma_{a+c-1})(n-1)+(\epsilon_a\ber\bullet \epsilon_b\ber)(n)=n a(\epsilon_{b}\ber_c\bullet \epsilon_a\sigma_{b+c-1})(n-1)+(\epsilon_b\ber\bullet \epsilon_a\ber)(n)$$
   from which \eqref{eq22} follows.
\end{proof}
\begin{corollary} Let $a,b,c\in\mathbb{C}$ and $n\in\mathbb{N},$ such that $b^n-a^n+c(b^{n-1}-a^{n-1})\neq 0$. Then the $n$th Bernoulli polynomial satisfies
    \begin{equation}\label{eq23}
        \ber_c(n)=\frac{1}{b^{n-1}(b+c)-a^{n-1}(a+c)}
        \sum_{i=1}^{n}{n\choose i}\ber_c(n-i)\Big(a^{n-1-i} b^{i}\sigma_{a+c-1}(i)-a^{i} b^{n-1-i}\sigma_{b+c-1}(i)\Big)
    \end{equation}
\end{corollary}
\begin{proof}
    It is a consequence of \eqref{eq22}.
\end{proof}
Note that \eqref{eq23} is the Tuenter's identity \cite{tuenter} for $c=0$.
Taking $c=1$ in \eqref{eq23}, gives a class of formulas for the Bernoulli numbers as follows
 \begin{equation}\label{eq24}
        \ber(n)=\frac{1}{b^{n-1}(b+1)-a^{n-1}(a+1)}
        \sum_{i=1}^{n}{n\choose i}\ber(n-i)(-1)^i\Big(a^{n-1-i} b^{i}\sigma_{a}(i)-a^{i} b^{n-1-i}\sigma_{b}(i)\Big).
    \end{equation}
\begin{corollary}
The higher order Bernoulli polynomials and the power sums satisfy
\begin{equation}
\epsilon_a\ber^r_c\bullet \epsilon_b \ber^r_{a+c}=\epsilon_b\ber^r_c\bullet \epsilon_a \ber^r_{b+c};~b^r(\epsilon_{a}\ber^r_c\bullet \epsilon_b\sigma^r_{a+c-1})=a^r(\epsilon_{b}\ber^r_c\bullet \epsilon_a\sigma^r_{b+c-1}),
\end{equation}
for each positive integer $r,$ and $a,b,c\in\mathbb{C}$.
\end{corollary}
\begin{proof}
   It is easy to verify that $(\epsilon_x f)^r=\epsilon_x f^r,$ for all $f\in\s$ and $x\in\mathbb{C}$. Now the proof follows from \eqref{eq21}.
\end{proof}
Many symmetric identities are a consequence
of the  following simple result.
\begin{proposition}\label{th2}
Let $f\in\mathcal{A}$ and $F=I\bullet f$. Then, for $g_1,~g_2\in\mathcal{A}$
\begin{equation}\label{eq30}
    f\bullet g_1=F\bullet g_2~\Leftrightarrow~g_1=I\bullet g_2.
\end{equation}
\end{proposition}
\begin{proof}
    It follows from $f\bullet g_1=f\bullet I\bullet g_2$.
\end{proof}
Note that, if $f,g\in\mathcal{A}$ such that $F=I\bullet f$ and $G=I\bullet g,$ then $f\bullet G=g\bullet F$.
\begin{example}
    Take $g_1(s)=\frac{1}{(m+n+s+1){m+n+s\choose m}}$ and $g_2(s)=\frac{(-1)^s}{(m+n+s+1){m+n+s\choose n}}$. Then  $g_1,~g_2$ satisfy \eqref{eq30}. Therefore, we get
\begin{equation}\label{eq31}
    \sum_{i=0}^{s}\frac{{s\choose i} f(i)}{(m+n+s-i+1){m+n+s-i\choose m}}=\sum_{i=0}^{s}\frac{{s\choose i} (-1)^{s-i}F(i)}{(m+n+s-i+1){m+n+s-i\choose n}},
\end{equation}
which is the identity of Gould and Quaintance \cite{gould}.

Similarly, taking $(g_1,g_2)=(\nu\ber,\ber)$ in \eqref{eq30} gives
\begin{equation}\label{eq32}
    \sum_{i=0}^{s}{s\choose i} (-1)^{s-i}\ber(s-i) f(i)=\sum_{i=0}^{s}{s\choose i} \ber(s-i)F(i).
\end{equation}
It can be easily checked that the pair $(g_1,g_2)=(\nu\eulr_1,\eulr_1)$ also satisfies \eqref{eq30}.

Moreover, if  $g\in\mathcal{A}$ satisfies $\nu g=\nu\bullet g,$ then the pair $(g_1,g_2)=(\nu g,g)$ satisfies \eqref{eq30}.
\end{example}
From the above discussion, one notices that
the Cauchy-type product is involved in many identities satisfied by the Faulhaber polynomials, the Bernoulli numbers, the higher order  Bernoulli numbers, the Bernoulli polynomials, etc. This motivates us to explore some algebraic properties of the Cauchy-type product, which is useful in
understanding the interplay between the identities. The Cauchy-type product on the arithmetic functions is essentially the same as the ordinary product on their respective generating functions. However, sometimes as we have seen earlier, it is advantageous to work with the algebraic properties of the Cauchy-type product in comparison to the usual product of the generating functions.
\section{Algebraic characterization}\label{sec:3}
The set $\mathcal{D}$ of all $f:\mathbb{N}\rightarrow \mathbb{C},$ such that $f(1)\neq 0,$ has an abelian group structure with respect to the Dirichlet multiplication $*$ defined by $$(f*g)(k)=\sum_{d|k}f(d)g\left(\frac{k}{d}\right).$$ The pair $(\mathcal{D},*)$ is a torsion-free group. So is the
pair $(\mathcal{A},\bullet),$ as proved below.
\begin{proposition}
    The group $(\mathcal{A},\bullet)$ is torsion-free.
\end{proposition}
\begin{proof}
If possible, suppose that $f\in\mathcal{A}$ such that $f\neq e$ and $f^s=e,$ for some positive integer $s>1$.
Then $f^{s}(0)=1$ and $f^s(k)=0$ for all $k>0$. To arrive at contradiction, we use induction on $k$ and show that $f(k)=0$ for all $k=1,2,\dots,$ i.e., $s=1$.

Consider $f^s(1)=0,$ which gives $\sum_{\sum_{i=1}^{s}k_i=1}{1\choose k_1,\ldots, k_s}f(k_1)\cdots f(k_s)=sf(1)f(0)^{s-1}=0$ or $f(1)=0$ since $s>0$ and $f(0)\neq 0$. Now suppose that $f(i)=0$ for all $i=1,\ldots,k$ and consider $f^s(k+1)=0,$ which gives $\sum_{\sum_{i=1}^{s}k_i=k+1}{k+1\choose k_1,\ldots,k_s}f(k_1)\cdots f(k+1)=0$ where the former expression
survives only when none of the $k_j$'s takes value from the set $\{1,\ldots,k\}$ in accordance with the induction hypothesis.
Therefore, $sf(k+1)f(0)^{s-1}=0$  or $f(k+1)=0$ thus proving the final step of induction. This shows that $f(k)=f(0)e(k)$. Also, $f^s(0)=f(0)^s e(0)=1,$ or $f(0)=1$. It follows that $f=e,$ which is a contradiction.
\end{proof}
The Cauchy product and the Cauchy-type product on the set $\mathcal{A}$ give rise to
the same group structure up to isomorphism.
\begin{theorem}
 $(\mathcal{A},\circ)\cong(\mathcal{A},\bullet)$.
 \end{theorem}
\begin{proof}
 Let $\xi\in\mathcal{A}$ be such that $\xi(k)=k!$ for all $k\in\A$.
 Note that ${k\choose m}=\frac{\xi(k)}{\xi(m)\xi(k-m)}$.
 Recall that, an ordinary product of two arithmetic functions $f$ and $g$ is the arithmetic function $fg$
defined by $(fg)(k)=f(k)g(k)$ for all $k\in\A$.
Now for any $f,g\in\mathcal{A},$ $fg\in\mathcal{A}$ since $(fg)(0)=f(0)g(0)\neq 0$. With these, consider for any $f,g\in\mathcal{A}$
\begin{equation}\label{eq25}
\begin{split}
    \frac{(\xi f\bullet \xi g)}{\xi}(k)&=\frac{1}{\xi(k)}\sum_{m=0}^{k}\frac{\xi(k)}{\xi(m)\xi(k-m)}(\xi f)(m)(\xi g)(k-m)\\
    &=\sum_{m=0}^{k}\frac{1}{\xi(m)\xi(k-m)}\xi(m) f(m)\xi(k-m) g(k-m)\\
    &=\sum_{m=0}^{k}f(m)g(k-m)=(f\circ g)(k)~\text{for all}~k\in\A.
    \end{split}
\end{equation}
Similarly, the reverse identity $\displaystyle \xi\left(\frac{f}{\xi}\circ\frac{g}{\xi}\right)=f\bullet g$ also holds. This association defines a bijection
$\Phi_\xi:(\mathcal{A},\circ)\rightarrow (\mathcal{A},\bullet)$ sending $f\mapsto \xi f$ s.t. $\Phi_\xi(f\circ g)=\xi (f\circ g)=(\xi f)\bullet (\xi g)=\Phi(f)\bullet \Phi(g)$. Thus $\Phi_\xi$ is an isomorphism.
 \end{proof}
 \begin{remark}
 The proof of the preceding theorem works even when the Cauchy product $\circ$ is
 replaced by any other product $\circ_1,$ for which there is a $\theta\in\mathcal{A},$ such
 that \eqref{eq25} holds, for $\circ$ replaced by $\circ_1,$ and $\xi$ replaced by $\theta$.

 In a similar way,  $(\mathcal{D},*)\cong(\mathcal{D},\star)$ where $\star$ is defined by
 \begin{equation}\label{eq26}
(f\star g)(k)=\sum_{d|k}\frac{\gamma(k)}{\gamma(d)\gamma\left(\frac{k}{d}\right)}f(d)
g\left(\frac{k}{d}\right)=\left(\gamma\left\{\frac{f}{\gamma}*\frac{g}{\gamma}\right\}\right)(k),
 \end{equation}
for some $\gamma\in\mathcal{D}$.
 \end{remark}
 \begin{example}
    If $k=\prod_{p|k}p^{\alpha_p(k)}$ is the prime power decomposition of $k\in\mathbb{N},$
then \eqref{eq26} is the binomial convolution for $\gamma\in\mathcal{D}$ such that $\gamma (\prod_{p|k}p^{\alpha_p(k)})=\prod_{p|k}\alpha_p(k)!,$ where we note that $\alpha_p(\frac{k}{d})=\alpha_p(k)-\alpha_p(d)$.
 \end{example}
The set of all completely multiplicative functions forms a subgroup of $\mathcal{D}$ with the binomial-convolution, which is not the case under Dirichlet convolution. For an account of the binomial convolution, we recommend the recent work of T\'oth and Haukkanen \cite{toth1}.

Dirichlet multiplication is a powerful tool in the multiplicative number theory.
Let $U_1:=\left\{f\in\mathcal{D}~|~f(1)=1\right\},$  $C_1:=\{ce~|c\neq 0,~c\in\mathbb{C}\},$ $U_M$ is the subgroup of $\mathcal{D}$ which consists of all multiplicative functions, and $U_A$ is the subgroup consisting of all anti-multiplicative functions. Then $\mathcal{D}=U_M\oplus U_A\oplus C_1,~U_1=U_M\oplus U_A$ (see Denlay \cite{denlay}).
We will prove an analogue of the aforementioned direct-sum decomposition for $(\mathcal{A},\bullet)$.

\begin{proposition}
    Following four sets are subgroups of $(\mathcal{A},\bullet)$
    \begin{enumerate}
\item $U:=\{f\in\mathcal{A}~|~f(0)=1\}$

\item $C:=\{f(0)e~|~f\in\mathcal{A}\}$

\item $V:=\{f\in \mathcal{A}~|~f(k_1+k_2)=f(k_1)f(k_2)\}$ for all $k_1,k_2\in\A$

\item $W:=\{f\in\mathcal{A}~|~f(0)=1,~f(1)=0\}$.
\end{enumerate}
\end{proposition}
\begin{proof}
    It follows from the subgroup-criterion.
\end{proof}
Observe that the set $V,$ as defined above, is a subgroup of $\mathcal{A}$ in the Cauchy-type product which is not the case in the usual Cauchy product. The subgroups $V$ and $W$ of $\mathcal{A}$ are analogous to the subgroups $U_M$ and $U_A$ of $\mathcal{D},$ respectively.
\begin{theorem}
    The group $(\mathcal{A},\bullet)$ has a direct sum decomposition given by $\mathcal{A}=U\oplus C$ where  $U=V\oplus W$.
\end{theorem}
\begin{proof}
Note that for any $f\in \mathcal{A},$ $g=\frac{f}{f(0)}\in U,$ where $g(k)=\frac{f(k)}{f(0)}$ for all $k\in\A$. Then $(\frac{f}{f(0)}\bullet f(0)e)(0)=f(0)$ for all $k\in \A$ and, $(\frac{f}{f(0)}\bullet f(0)e)(k)=\sum_{m=0}^k \frac{f(m)}{f(0)}f(0)e(k-m)=f(k)e(0)=f(k)$ for all $k>0$. Thus $f=\frac{f}{f(0)}\bullet f(0)e$ for all $f\in\mathcal{A}$. So, $\mathcal{A}=U\oplus C$ since $U\cap C=\{e\}$.

For the remaining part,  observe that $e\in V$ so $V\neq\emptyset$. Moreover $f(k)=f(1)^k$ for all $k\in\A$ so that $f$ is completely determined by its value at $1$. Therefore $f(k)=0$ if and only if $f(1)=0$ if and only if $f=e$. Now for any $f,g\in V,$ $f(0)=1=g(0)$ and $f\bullet g^{-1}=\epsilon_{f(1)+\frac{1}{g(1)}}\in V;$ therefore $V\leq U$.
Similarly, $W\neq\emptyset$ as $e\in W$. Also, for $f,g\in W,$ $(f\bullet g^{-1})(0)=f(0)g^{-1}(0)=1$ and $(f\bullet g^{-1})(1)=f(1)+g^{-1}(1)=-\frac{g(1)}{g(0)}=0$. Thus $f\bullet g^{-1}\in W$. This verifies that $W\leq U$.
Now for any $f\in U$ if $f(1)\neq 0$ then
 $f=f_1\bullet f_2\in V\oplus W$ where $f_1(k)=(-f(1))^k$ and $f_2(k)=(\epsilon_{f(1)}\bullet f)(k)$.
 Clearly $f_1\in V$ and, $f_2(1)=f(1)+(-f(1))^1 f(0)=0,~f_2(0)=f(0)=1$. Therefore $f_2\in W$. On the other hand if $f(1)=0$ then
 $f=e\bullet f\in V\oplus W$. So, $U=V\oplus W$.
\end{proof}
\subsection{Vector space structure}
Since $(\mathcal{A},\bullet)$ is torsion-free, it can be regarded as a vector space over $\mathbb{Q}$ with the scalar multiplication $\cdot:\mathbb{Q}\times \mathcal{A}\rightarrow \mathcal{A}$  defined  by $$\frac{p}{q}\cdot f:=f^{\frac{p}{q}}$$ where $f^{\frac{p}{q}}:=g\in\mathcal{A}$ such that $g^q=f^p$ with the obvious notation $f^0=f;$ $f^{-k}:=(f^{k})^{-1}$ and $f^k={f\bullet \cdots \bullet f}(k-$times) for $k\in\mathbb{N}$. Such a $g$ is unique for a given $f$.
To see this, if $h^q=g^q$ for some $h\in\mathcal{A}$ and $q\in\mathbb{N}$ then $(h\bullet g^{-1})^{q}=e$ which gives $g=h$ since otherwise $h\bullet g^{-1}$ will be a non identity  element of finite order in $\mathcal{A},$ contradicting the fact that $\mathcal{A}$ is torsion-free. This observation allows us to calculate for any $m\in\mathbb{N},$ the $m$th root of any $f\in \mathcal{A}$, i.e., an arithmetic function $g\in\mathcal{A}$ such that $g^m=f,$ which can be computed inductively via the following:
\begin{equation}\label{eq27}
\begin{split}
    g(0)=&f(0)^{\frac{1}{m}},~g(1)=\frac{f(1)}{mg(0)^{m-1}},\\
   g(k)=&\frac{1}{m g(0)^{m-1}}\left\{f(k)-\sum_{\sum_{i=1}^{m}k_i=k,~k_i<k}
   {k\choose k_1,\ldots,k_m}g(k_1)\cdots g(k_m)
\right\},~k\geq 2.
\end{split}
\end{equation}
\begin{example}
    If we take $f=\epsilon_x$ and $m=2$ then $g=\epsilon_{-x}^{1/2}$
can be calculated using the preceding formula \eqref{eq27} as follows
$$g(0)=1;~g(1)=-\frac{x}{2};~g(2)=\frac{1}{2}(f(2)-2g(1)^2)=\frac{x^2}{4},$$
$$~g(3)=\frac{1}{2}(-x^3-6g(1)g(2))=-\frac{x^3}{8};~g(4)=\frac{x^4}{4}$$
and  inductively leads to $\epsilon_{-x}^{1/2}=\epsilon_{-x/2}$.
More generally, for any rational $r\in\mathbb{Q},$  one has $\epsilon_{x}^{r}=\epsilon_{rx}$.
\end{example}
Since each $f\in V$ is of the form $f=\epsilon_{f(1)},$ it follows that $V$ is a subspace of $U$. It is also clear that, for $m\in\mathbb{N}$ and $f\in W,$ the $m$th root of $f$ is in  $W$. Therefore, $f^{r}\in W$ for all $r\in\mathbb{Q}$. Thus $W$ is also a subspace of $U$. Note that a nontrivial subgroup of a torsion-free group is torsion-free; therefore, each of the nontrivial subgroups of $\mathcal{A}$ can be regarded as a vector space over $\mathbb{Q}$.

Also, it is not hard to see that, for a Hamel-basis $\mathcal{H}$ of $\mathbb{C}(\mathbb{Q}),$  the collection of arithmetic functions $\mathcal{B}_{V}:=\{\epsilon_{x}~|~ x\in \mathcal{H}\}$ is a Hamel-basis of the vector space $V(\mathbb{Q})$.  Therefore, $V(\mathbb{Q})\cong \mathbb{C}(\mathbb{Q})\cong U/W$. On these lines, it will be interesting to obtain an appropriate Hamel-basis for the rational vector space $\mathcal{A}$.

Table \ref{t1} gives the $m$th roots of the Bernoulli numbers as evaluated using \eqref{eq27}, for some values of $m$ and $k$.
 The numerator and the denominator of $\ber(k)^{\frac{1}{2}},~k=0,1,\ldots,$ correspond to the sequences A241885 and A242225, respectively, in the \emph{On-Line Encyclopedia of Integer Sequences} \cite{sloan}.
\begin{table}[ht!]
\centering\caption{Some roots of the Bernoulli numbers in $(\mathcal{A},\bullet)$}\label{t1}
\vspace{10pt}
\begin{tabular}{c|ccccccccc}
  \hline
  $k$ & 0 & 1 & 2 & 3 & 4 &5&6&7&8\\
  \hline\\
  $\ber^{\frac{1}{2}}(k)$ & 1 & $-\frac{1}{4}$ & $\frac{1}{48}$ & $\frac{1}{64}$ & $-\frac{3}{1280}$&$-\frac{19}{3072}$&$\frac{79}{86016}$&$\frac{275}{49152}$
  &$-\frac{2339}{2949120}$\\~\\

  $\ber^{\frac{1}{3}}(k)$ & 1 & $-\frac{1}{6}$ & $\frac{1}{54}$ & $\frac{7}{324}$ & $\frac{2}{3645}$&$-\frac{197}{13122}$&$-\frac{683}{61236}$&$\frac{1009}
  {59049}$
  &$\frac{261203}{5314410}$\\~\\

  $\ber^{\frac{1}{4}}(k)$ & 1 & $-\frac{1}{8}$ & $\frac{7}{384}$ & $\frac{39}{2048}$ & $-\frac{2311}{491520}$&$-\frac{9471}{524288}$&$\frac{254713}{176160768}$&
  $\frac{16744565}
  {402653184}$
  &$\frac{1127877731}{96636764160}$\\~\\
  $\ber^{\frac{1}{5}}(k)$ & 1 & $-\frac{1}{10}$ & $\frac{13}{750}$ & $\frac{97}{6250}$ & $-\frac{237}{31250}$&$-\frac{69061}{4687500}$&$\frac{9768883}{820312500}$&
  $\frac{99676471}
  {2929687500}$
  &$-\frac{827331922}{18310546875}$\\~\\
 % &&&&&&&&&\\
  \hline
\end{tabular}
\end{table}

The sequence $\ber^{\frac{p}{q}}(k),~q\neq 0,~p,q\in\mathbb{N}$ is determined by the generating function
\begin{equation}\label{eq28}
    \Bigl(\frac{t}{e^t-1}\Bigr)^{\frac{p}{q}}:=\sum_{k=0}^{\infty}\ber^{\frac{p}{q}}(k)\frac{t^k}{k!},
\end{equation}
which shows that $\ber^{\frac{p}{q}}(k)$ is the $k$th N\"orlund polynomial (see Liu and Srivastava \cite{srivastava2006}) in $\frac{p}{q}$. For example, $$\ber^{\frac{p}{q}}(0)=1;~\ber^{\frac{p}{q}}(1)=-\frac{p}{2q};~\ber^{\frac{p}{q}}(2)=\frac{p(3p-q)}{12 q^2},$$ $$\ber^{\frac{p}{q}}(3)=-\frac{p^2(p-q)}{8 q^3},~\ber^{\frac{p}{q}}(4)=\frac{p(15 p^3-30p^2q+5pq^2+2q^3)}{240q^4},$$ etc.
\section{Mixed products}\label{sec:4}
Singh \cite{singh2013} introduced the M\"obius-Bernoulli polynomials via the following:
$$M_k(x,n)=\sum_{d|n}\mu(d) d^{k-1} \ber_{\frac{x}{d}}(k),~k\in\A,$$
where the $k$th M\"obius-Bernoulli polynomial is denoted by $M_k(x,n)$ and $n\in\mathbb{N}$.
For $x\in\mathbb{C}$ and $n\in\mathbb{N},$ let $\mathcal{M}_{x,n}\in\mathcal{A}$ such that $\mathcal{M}_{x,n}(k):=M_k(x,n)$. Then
 $$\mathcal{M}_{0,n}(k)=\ber(k)\sum_{d|n}\mu(d)d^{k-1}=\ber(k)(\mu*N_{k-1})(n)$$ where and $N_k(n):=\epsilon_n(k)$ for all $k\in\A$ and $N_{-1}(n):=\frac{1}{n}$.
 The $k$th M\"obius-Bernoulli number is defined as $M_k(n)=M_{0,n}(k)$.
 Now the following remarkable identity holds
\begin{equation}\label{eq29}
    \sum_{i=1;~\gcd(i,n)=1}^{n}i^k=\xi_1(k)(\mathcal{M}_{0,n}\bullet (\epsilon_{n+1}-\epsilon_1))(k)=(\mu N_{k-1}*S_{k,x})(n)
\end{equation}
where $S_{k,x}(n):=\mathcal{S}_{x/n}(k)$.
Note that the identity \eqref{eq29} involves both the Dirichlet convolution and the Cauchy-type product.
It may be useful to explore
the identities expressing
an interaction between the Dirichlet product and the Cauchy-type product.
\section{Acknowledgements}
The author is indebted to Professor Jeffery O. Shallit (Chief-Editor of this Journal), the unknown referees, and my colleagues Parminder Singh and Harpreet Kaur, for their helpful suggestions in the linguistic-improvement of the paper.

\bibliographystyle{plain}

\bigskip
\hrule
\bigskip

\noindent 2010 {\it Mathematics Subject Classification}:
Primary 11A25; Secondary 11B68,  05A10, 11B65.

\noindent \emph{Keywords: }
Cauchy product, Cauchy-type product, Dirichlet convolution, arithmetic function, Bernoulli number, torsion-free group, Bernoulli polynomial, power sum.

\bigskip
\hrule
\bigskip

\noindent (Concerned with sequences
{A027641} {A027642} {A116419} {A116420} {A241885} and A242225.)

\end{document}